\documentclass[a4paper,11pt]{article}
\usepackage[]{graphicx}
\usepackage[]{amsmath,amssymb}
\usepackage[]{amsthm,enumerate}
\usepackage{verbatim,bm}
\usepackage{color}
\usepackage{a4wide}
\usepackage{braket}
\usepackage{array}

\newtheorem{theorem}{Theorem}[section]
\newtheorem{lemma}[theorem]{Lemma}

\theoremstyle{definition}

\newtheorem{example}[theorem]{Example}

\theoremstyle{remark}
\newtheorem{remark}[theorem]{Remark}

\DeclareMathOperator{\Spec}{Spec}
\newcommand{\bhline}[1]{\noalign{\hrule height #1}}
\newcolumntype{I}{!{\vrule width 2pt}}

\title{Mutually orthogonal Sudoku Latin squares and their graphs}
\author{
Sho Kubota\thanks{Department of Applied Mathematics, Faculty of Engineering,
Yokohama National University, 
Yokohama, Kanagawa 240-8501, Japan. 
\texttt{kubota-sho-bp@ynu.ac.jp}}
\and
  Sho Suda\thanks{Department of Mathematics,  National Defense Academy of Japan, Yokosuka, Kanagawa 239-8686, Japan. \texttt{ssuda@nda.ac.jp}}
\and 
Akane Urano\thanks{Konan Highschool,  Konan, Aichi 483-8177, Japan. }
}
\date{\today}

\begin{document}
\maketitle
\abstract{
We introduce a graph attached to mutually orthogonal Sudoku Latin squares. 
The spectra of the graphs obtained from finite fields are explicitly determined. 
As a corollary, we then use the eigenvalues to distinguish non-isomorphic Sudoku Latin squares. 
}

%%%%%%%%%%%%%%%%%%%%%%%%%%%%%%%%%%%%%%%%%%%%%%%%%%%%%%%%%%%%%%%%%%%%%%%%%%%%%%%%%%%%%%%%%%%%%%
\section{Introduction}
A Latin square is an array of order $n$ with $n$ symbols such that each symbol occurs exactly once in each row and column.  
Orthogonal Latin squares are a pair of Latin squares of same order such that the superimposition of the same entries of the Latin squares contains each ordered pair. 
Mutually orthogonal Latin squares (MOLS) are a family of Latin squares such that any two of them are orthogonal.  
One of the most important problem on the area of Latin squares is to determine the maximum number $N(n)$ of Latin squares of order $n$ for which they consist of mutually orthogonal Latin squares.  
It is easy to see that the upper bound on $N(n)$ is at most $n-1$, and the upper bound is achieved when $n$ is a prime power (see \cite{DK} for example). 
The converse is conjectured, that is, if $N(n)=n-1$, then is $n$ a prime power?   

In the last decades, Sudoku Latin squares have been investigated in the literature. 
A Sudoku Latin square is a Latin square of order $n^2$ such that the array is divided into $n^2$ $n\times n$ subarrays such that each symbol occurs exactly once in each subarray.  
Mutually orthogonal Sudoku Latin squares (MOSLS) have been investigated in the topics such as the upper bounds and maximal examples obtained from finite fields \cite{BCC,PV}, non-extendible examples \cite{DMSV}. 

In this paper, we consider a slightly generalized case, which we call {\it Sudoku Latin square of order $qr$ of type $(q,r)$}: Latin squares of order $qr$ partitioned into $qr$ $q\times r$ subarrays such that each symbol occurs exactly once in a subarray.

It is well-known that MOLS $L_1,\ldots,L_f$ of order $n$ yield a strongly regular graph, denoted $G_{[f],\text{MOLS}}$, with parameters $(n^2,(f+2)(n-1),n-2+f(f+1),(f+1)(f+2))$, see Section~\ref{sec:mols}.   
In particular, the spectra of the graph $G_{[f],\text{MOLS}}$ is uniquely determined by the order and the number of Latin squares consisting of MOLS. 

In this paper, we define a graph, denoted $G_{[f],\text{MOSLS}}$, attached to MOSLS by adding edges corresponding to subarrays in Sudoku Latin squares to the graph attached to MOLS. 
We are interested in the spectrum of the graph $G_{[f],\text{MOSLS}}$. 
In contrary to the case of MOLS, the spectra are not uniquely determined by the order and the number of Sudoku  Latin squares consisting of MOSLS in general. 
We explicitly determine the spectrum of the graph of MOSLS with certain property which MOSLS obtained from finite fields satisfy. 
Also we apply a cycle switching to a Sudoku Latin square obtained from finite fields to obtain a new Sudoku Latin square with different spectra. 
Thus we conclude that the new Sudoku Latin square is not equivalent to the old one, and the spectra of the graphs obtained from a Sudoku Latin square can be used to distinguish non-equivalent Sudoku Latin squares.

%%%%%%%%%%%%%%%%%%%%%%%%%%%%%%%%%%%%%%%%%%%%%%%%%%%%%%%%%%%%%%%%%%%%%%%%%%%%%%%%%%%%%%%%%%%%%%
\section{Preliminaries}
\subsection{Graphs}

A {\it graph} $X$ consists of a vertex set $V(X)$ and an edge set $E(X)$,
where $E(X) \subset \binom{V(X)}{2} = \{ \{x,y\} \mid x,y \in V(X), x \neq y \}$.
The {\it adjacency matrix} $A = A(X)$ of a graph $X$ is
the integer matrix with rows and columns indexed by $V(X)$ such that
\[ A_{x,y} = \begin{cases}
1 \quad &\text{if $\{x,y\} \in E(X)$,} \\
0 \quad &\text{otherwise.}
\end{cases} \]
The {\it spectrum} of a graph $X$, denoted by $\Spec (X)$ or $\Spec(A(X))$, is the list of its eigenvalues of $A(X)$ together with their multiplicities.
Define $N_{X}(x) = \{ y \in V(X) \mid \{x,y\} \in E(X) \}$ for a vertex $x$ of a graph $X$.
If there is no fear of confusion,
we simply write $N(x)$ instead of $N_{X}(x)$.
A graph $X$ is said to be {\it $k$-regular} if $|N(x)| = k$ for any $x \in V(X)$.

For example,
we consider the graph $X$ defined by $V(X) = \{1,2,3\}$ and $E(X) = \{\{1,2\}, \{2,3\}\}$.
The adjacency matrix is
\[ \begin{bmatrix}
0&1&0 \\
1&0&1 \\
0&1&0
\end{bmatrix}, \]
so the spectrum is $\Spec (X) = \{ [\sqrt{2}]^{1}, [0]^{1}, [-\sqrt{2}]^{1} \}$,
where the superscripts indicate the multiplicities of eigenvalues.

Let $X$ be a graph,
and let $\pi = \{C_1, C_2, \dots, C_t \}$ be a partition of $V(X)$.
The partition $\pi$ is {\it equitable} if
$|N(x) \cap C_j |$ is a constant number, say $q_{i,j}$, for a vertex $x \in C_i$.
Then, the matrix $Q \in \mathbb{C}^{\pi \times \pi}$ defined by $Q_{i,j} = q_{i,j}$
is called the {\it quotient matrix} with respect to $\pi$.

\begin{lemma}{\rm (\cite[Lemma~2.3.1]{BH})} 
Let $X$ be a graph with an equitable partition $\pi$,
and let $Q$ be the quotient matrix with respect to $\pi$.
Then we have $\Spec(Q) \subset \Spec(X)$.
\end{lemma}

Let $X$ be a regular graph such that $E(X) \neq  \emptyset, \binom{V(X)}{2}$.
Then $X$ is said to be {\it strongly regular with parameters} $(n, k, \lambda, \mu)$
if $|V(X)| = n$ and for any $x, y \in V(X)$,
\[
|N(x) \cap N(y)| = \begin{cases}
k \quad &\text{if $x = y$,} \\
\lambda \quad &\text{if $\{x,y\} \in E(X)$,} \\
\mu \quad &\text{if $\{x,y\} \not\in E(X)$ and $x \neq y$.}
\end{cases}
\]

\begin{lemma}\label{lem:srgev}{\rm (\cite[Chapter~10]{GR})}
Let $X$ be a strongly regular graph with parameters $(n, k, \lambda, \mu)$.
Let $\Delta = (\lambda-\mu)^2 + 4(k-\mu)$.
Then
\[\Spec(X) = \left\{
[k]^{1}, \left[ \frac{\lambda - \mu + \sqrt{\Delta}}{2} \right]^s,
\left[ \frac{\lambda - \mu - \sqrt{\Delta}}{2} \right]^t
\right\}, \]
where
\begin{align*}
s &= \frac{1}{2}\left( (n-1) - \frac{2k + (n-1)(\lambda - \mu)}{\sqrt{\Delta}}  \right), \\
t &= \frac{1}{2}\left( (n-1) + \frac{2k + (n-1)(\lambda - \mu)}{\sqrt{\Delta}}  \right).
\end{align*}
\end{lemma}

%\begin{itemize}
%\item Graph, adjacency matrix
%\item Equitable partition
%\end{itemize}

\subsection{Mutually orthogonal Latin squares and their graphs}\label{sec:mols}
For a positive integer $n$, set $[n]=\{1,2,\ldots,n\}$. 
A {\it Latin square of order $n$ on $[n]$} is an $n\times n$ array with entries in the set $[n]$ of symbols $1,2,\ldots,n$ such that each row and each column contain every symbol exactly once. 
Two Latin squares $L_1,L_2$ of order $n$ on $[n]$ are said to be {\it orthogonal} if the ordered pairs $(L_1(i,j),L_2(i,j))$ ($i,j\in[n]$) are distinct where $L(i,j)$ denotes the $(i,j)$-entry of an array $L$.   
Latin squares $\{L_1,L_2,\ldots,L_f\}$ are said to be {\it mutually orthogonal} if any two distinct Latin squares of them are orthogonal.  

\begin{example}
The following are an example of orthogonal Latin squares of order $4$:  
$$
L_1=\begin{array}{|c|c|c|c|}\hline
 1 & 2 & 4 & 3 \\ \hline
 3 & 4 & 2 & 1 \\ \hline
 4 & 3 & 1 & 2 \\ \hline
 2 & 1 & 3 & 4 \\ \hline
\end{array},\quad L_2=\begin{array}{|c|c|c|c|}\hline
 1 & 4 & 3 & 2 \\ \hline
 3 & 2 & 1 & 4 \\ \hline
 4 & 1 & 2 & 3 \\ \hline
 2 & 3 & 4 & 1 \\ \hline
\end{array}.
$$
\end{example}

Let $L_k$ ($k\in [f] %\{1,\ldots,f\}
$) be mutually orthogonal Latin squares of order $n$ on the set $[n]$ of symbols $1,2,\ldots,n$. 
For a subset $F$ of $[f]$, consider the set 
$$
C_{F}=\{(i,j,L_k(i,j))_{k\in F} \mid 1\leq i,j\leq n\},  
$$
where $L_k(i,j)$ denotes the $(i,j)$-entry of $L_k$. 
Regard the set $C_F$ as a subset of $[n]^{|F|+2}$ and consider the graph $G_{F,\text{MOLS}}=(C_F,E)$ where 
$$
E_{F,\text{MOLS}}=\{\{x,y\} \mid x,y\in C_{F},  d(x,y)=|F|+1\}
$$
where $d(x,y)$ is the Hamming distance between $x$ and $y$. Write its adjacency matrix as $A_{F,\text{MOLS}}$.
When $F=\{m\}$, we denote by $A_{m,\text{MOLS}}$ in stead of $A_{\{m\},\text{MOLS}}$. 

\begin{lemma}\label{lem:srgmols}{\rm (\cite[Theorem~10.4.2]{GR})}
Let $L_k$ {\rm (}$k\in [f] ${\rm }) be mutually orthogonal Latin squares of order $n$ on the set $[n]$ of symbols $1,2,\ldots,n$. 
Then the graph $G_{[f],\text{MOLS}}$ is a strongly regular graph with parameters $(n^2,(f+2)(n-1),n-2+f(f+1),(f+1)(f+2))$.  
\end{lemma}

\section{Mutually orthogonal Sudoku Latin squares}\label{sec:graph}

\subsection{Mutually orthogonal Sudoku Latin squares and some properties}
For positive integers $m$ and $n$ with $m<n$, define $[m,n]=\{x\in\mathbb{Z} \mid m\leq x \leq n\}$. 
A Latin square of order $qr$ is said to be a {\it Sudoku Latin square of type $(q,r)$} 
 if each submatrix of $L$ restricted the rows to $[(i-1)q+1,iq]$ 
 and the columns to $[(j-1)r+1,jr]$, called the {\it $(i,j)$-block},  
for $i\in[q],j\in[r]$ contains every symbol exactly once. 
When the order of a Sudoku Latin square is a square $q^2$ and we omit its type, then it means that the type is $(q,q)$.
Though Sudoku Latin squares of types $(q,1)$ or $(1,r)$ are exactly the same as just Latin squares, the cases of $q=1$ or $r=1$ are allowed. 
Define a {\it row-block} as the rows with indices in $[(i-1)q+1,iq]$ for some $i$, and define a {\it column-block} similarly.

\begin{example}
The following is an example of Sudoku Latin square of order $9$ of type $(3,3)$: 
$$
L=
\begin{array}{Ic|c|cIc|c|cIc|c|cI} \bhline{2pt}
 5 & 6 & 4 & 8 & 9 & 7 & 2 & 3 & 1 \\ \hline
 9 & 7 & 8 & 3 & 1 & 2 & 6 & 4 & 5 \\ \hline
 1 & 2 & 3 & 4 & 5 & 6 & 7 & 8 & 9 \\ \bhline{2pt}
 3 & 1 & 2 & 6 & 4 & 5 & 9 & 7 & 8 \\ \hline
 4 & 5 & 6 & 7 & 8 & 9 & 1 & 2 & 3 \\ \hline
 8 & 9 & 7 & 2 & 3 & 1 & 5 & 6 & 4 \\ \bhline{2pt}
 7 & 8 & 9 & 1 & 2 & 3 & 4 & 5 & 6 \\ \hline
 2 & 3 & 1 & 5 & 6 & 4 & 8 & 9 & 7 \\ \hline
 6 & 4 & 5 & 9 & 7 & 8 & 3 & 1 & 2 \\ \bhline{2pt}
\end{array}.$$

\end{example}

In the present paper, we focus on the following certain property of Sudoku Latin squares. 
Let $B_{i,j}=\{(a,b)\in\mathbb{N}^2 \mid (i-1)q+1\leq a\leq iq, (j-1)r+1\leq b\leq jr \}$ for $i\in[r],j\in[q]$. 
For a Sudoku Latin square $L$ order $qr$ of type $(q,r)$, we denote by $L[i,j]$ the subarray of $L$ with rows and columns restricted to $B_{i,j}$.  
A Sudoku Latin square $L$ of order $qr$ of type $(q,r)$ is said to be of {\it block-permutational} if for any $i,j$, there exist permutation matrices $P, Q$ of order $q,r$ respectively such that $L[i,j]=P^{-1}L[1,1]Q$ holds.   

\begin{lemma}\label{lem:adj2}
Let $M,M'$ be $q\times r$ matrices with distinct entries in $[qr]$, and $P,Q$ be permutation matrices of orders  $q,r$, respectively. 
Let $A$ be a $qr\times qr$ $(0,1)$-matrix with rows and columns indexed by $[q]\times[r]$ defined by $A_{(i,j),(i',j')}=1$ if and only if $M_{i,j}=(M')_{i',j'}$
Then $M_{i,j}=(P^{-1}MQ)_{i',j'}$ holds  if and only if $A=P\otimes Q$ holds.   
\end{lemma}
\begin{proof} 
Write $P=(\delta_{i,\sigma(j)})_{i,j=1}^q$ and $Q=(\delta_{i,\tau(j)})_{i,j=1}^r$ for some permutations $\sigma$  on $[q]$ and $\tau$ on $[r]$. 
%\begin{align*}
%A_{(i,j),(i',j')}=1 \Leftrightarrow M_{i,j}=(P^{-1}MQ)_{i',j'}\Leftrightarrow M_{i,j}=M_{\sigma(i'),\tau(j')}.
%\end{align*}
\begin{align*}
(P\otimes Q)_{(i,j),(i',j')}=1 \Leftrightarrow P_{i,i'}=Q_{j,j'}=1 \Leftrightarrow i=\sigma(i'),j=\tau(j'). %\Leftrightarrow  M_{i,j}=M_{\sigma(i'),\tau(j')} \Leftrightarrow M_{i,j}=(P^{-1}MQ)_{i',j'}.
\end{align*}

Since the entries of $M$ are mutually distinct, this condition is equivalent to the fact that $M_{i,j}=M_{\sigma(i'),\tau(j')}$, which is equivalent to the fact that $M_{i,j}=(P^{-1}MQ)_{i',j'}$.  
%Therefore it is shown that $A$ equals to $P\otimes Q$.   
\end{proof}

\subsection{Associated graphs and their adjacency matrices}
In this subsection we define a graph from mutually orthogonal  Sudoku Latin squares. 

Let $L_k$ ($k\in [f]%\{1,\ldots,f\}
$) be mutually orthogonal Sudoku Latin squares of order $qr$ on the set $[q^2r^2]$ of the  symbols $1,2,\ldots,q^2r^2$. 
For $F\subset [f]$, 
consider the graph $G_{F,\text{MOSLS}}=(C_{F},E_{F,\text{MOSLS}})$ with $E_{F,\text{MOSLS}}$ defined as
\begin{align*}
E_{F,\text{MOSLS}}=&E_{F,\text{MOLS}}\\
&\cup \{(x_1,x_2),(y_1,y_2)\} \in \tbinom{C_F}{2} |  
 (x_1,x_2),(y_1y_2)\in B_{i,j}\text{ for some }i,j.\}
\end{align*}

Write its adjacency matrix as $A_{F,\text{MOSLS}}$.
When $F=\{m\}$, we denote by $A_{m,\text{MOSLS}}$ in stead of $A_{\{m\},\text{MOSLS}}$. 

\begin{remark}
A {\it Sudoku graph} was defined in \cite{S} as the vertex set equal to a set of cells of the $n\times n$ grid and the edge set consisting of  pairs $(x,y)$ where $x,y$ are in a same row, in a same column, or in a same block. 
The spectrum of the Sudoku graph is explicitly determined. 
\end{remark}

It is easy to see that the spectra of graphs obtained from MOSLS of order $qr$ is invariant under the following symmetries:
\begin{itemize}
\item relabeling the symbols $[qr]$,  
\item permuting the row-blocks and the column-blocks, and the rows in a row-block and the columns in a column-block, 
\item any reflection or rotation type keeping the type. 
\end{itemize}

Here we label the rows and columns of the matrix as follows. 
Define a map 
We consider the bijection $[r]\times[q]\times [q]\times[r]$ to $[q^2r^2]$ by $(i,j,k,\ell)\rightarrow  qr((i-1)r+k-1)+(j-1)q+\ell$.

In this ordering of the rows and columns, set 
$$
B=I_{rq}\otimes (J_q-I_q)\otimes (J_r-I_r).
$$
Then the matrix %$A_{F,\text{MOSLS}}$ defined by $A_{F,\text{MOSLS}}=A_{F,\text{MOLS}}+B$ 
$A_{F,\text{MOLS}}+B$ coincides with the adjacency matrix $A_{F,\text{MOSLS}}$ of the graph of mutually orthogonal Sudoku Latin squares.  
Set $I_1=\{(i,j,k,\ell) \in [r]\times[q]\times [r]\times[q]\mid i\neq k, j\neq \ell\}$, $I_2=\{(i,j,k,\ell) \in [r]\times[q]\times [r]\times[q]\mid i=j\text{ and }k\neq\ell, \text{ or }i\neq j\text{ and }k=\ell\}$. 
For each $m\in F$, the following holds: 
\begin{align*}
A_{m,\text{MOLS}}&=\sum_{(i,j,k,\ell)\in I_1} e_{\varphi(i,j),\varphi(k,\ell)}\otimes P_{m,(i,j),(k,\ell)}\\
&\quad+\sum_{(i,j,k,\ell)\in I_2} e_{\varphi(i,j),\varphi(k,\ell)}\otimes (I_q\otimes J_{{r}}+P'_{m,(i,j),(k,\ell)})\\
&\quad+\sum_{i\in [r],j\in [q]} e_{\varphi(i,j),\varphi(i,j)}\otimes (I_q\otimes(J_{{r}}-I_{{r}})+(J_q-I_q)\otimes I_{{r}}),
\end{align*}
where $\varphi(a,b)=(a-1)q+b$ and $e_{i,j}$ is a square matrix of order $qr$ whose $(i,j)$-entry is $1$ and all other entries are $0$, 
$P_{m,(i,j),(k,\ell)}$ is a permutation matrix of order $qr$, and  
$P'_{m,(i,j),(k,\ell)}$ is a permutation matrix of order $qr$ such that $I_q\otimes J_r+P'_{m,(i,j),(k,\ell)}$ is a $(0,1)$-matrix.  
\begin{lemma}\label{lem:quotient}
Let $L_k$ {\rm(}$k\in [f]${\rm)} be mutually orthogonal  Sudoku Latin squares of order $qr$ of type $(q,r)$. 
The following is a subset of the spectrum of $G_{[f],\text{MOSLS}}$:  
\begin{align*}
\{[3qr-q-r-1+f(qr-1)]^1,[2qr-q-r-1-f]^{q+r-2}, [qr-q-r-1-f]^{(q-1)(r-1)}\}. 
\end{align*}
\end{lemma}
\begin{proof}
Recall that $B_{i,j}=\{(a,b)\in\mathbb{N}^2 \mid (i-1)q+1\leq a\leq iq, (j-1)r+1\leq b\leq jr \}$ for $i\in[r],j\in[q]$. 
The partition $\{B_{i,j}\mid i\in[r],j\in[q]\}$ of $[qr]^2$ is an equitable partition. 
The quotient matrix $Q$ of $A_{[f],\text{MOSLS}}$ with the ordering $B_{1,1},B_{1,2},\ldots,B_{1,q},\ldots,B_{r,q}$ is 
\begin{align*}
Q&=(qr-1)I_{rq}+(r+f)I_r\otimes (J_q-I_q)+(q+f)(J_r-I_r)\otimes I_q+f(J_r-I_r)\otimes(J_q-I_q)
\\
&=(qr-q-r-1-f)I_{rq}+r I_r\otimes J_q+q J_r\otimes I_q+fJ_{rq}.  
\end{align*}
It is easy to see that the spectrum of $Q$ is 
\begin{align*}
\{[3qr-q-r-1+f(qr-1)]^1,[2qr-q-r-1-f]^{q+r-2}, [qr-q-r-1-f]^{(q-1)(r-1)}\}. 
\end{align*}
Therefore, $A_{[f],\text{MOSLS}}$ has the desired eigenvalues. 
\end{proof}

The following lemma provides a sufficient condition for $A_{[f],\text{MOLS}}$ and $B$ to commute. 
\begin{lemma}\label{lem:1} 
Assume that each matrix $A_{m,\text{MOSLS}}$ satisfies that   
$
P_{m,(i,j),(k,\ell)}=P_{m,1,i,j}\otimes P_{m,2,k,\ell}
$ where $P_{m,1,i,j}$ and $P_{m,2,k,\ell}$ are permutation matrices, and  
$P'_{m,(i,j),(k,\ell)}=P'_{m,1,i,j}\otimes P'_{m,2,k,\ell}
$ where $P'_{m,1,i,j}$ and $P'_{m,2,k,\ell}$ are permutation matrices and $P'_{m,1,i,j}$ has $0$ on the diagonal entries.  
Then the matrices $A_{[f],\text{MOLS}}$ and $B$ commute. 
\end{lemma}
\begin{proof}
Since 

\begin{align*}
A_{[f],\text{MOLS}}&=\sum_{m\in [f]}A_{m,\text{MOLS}}-(f-1)I_{rq}\otimes (I_q\otimes(J_r-I_r)+(J_q-I_q)\otimes I_r)\\
&-(f-1)(J_q-I_q)\otimes I_r \otimes I_q\otimes J_r,
\end{align*} 
it is enough to verify that $A_{m,\text{MOLS}}$ and $B$ commute for each $m$. 
To do this, by 
\begin{align*}
B=I_{qr}\otimes(I_q\otimes I_r-I_q\otimes J_r-J_q\otimes I_r+J_q\otimes J_r),  
\end{align*}
we show that 
$K\in\{I_q\otimes J_r,J_q\otimes I_r\}$ and $P_{m,(i,j),(k,\ell)}$ or $P'_{m,(i,j),(k,\ell)}$ commute. 
Indeed, 
\begin{align*}
(I_q\otimes J_r)P_{m,(i,j),(k,\ell)}&=(I_q\otimes J_r)P_{m,(i,j),(k,\ell)}=(I_q\otimes J_r)(P_{m,1,i,j}\otimes P_{m,2,k,\ell})\\
&=P_{m,1,i,j}\otimes J_r P_{m,2,k,\ell}\\
&=P_{m,1,i,j}\otimes J_r.
\end{align*}

In the same manner, we have $P_{m,(i,j),(k,\ell)}(I_q\otimes J_r)=P_{m,1,i,j}\otimes J_r$, and thus  $(I_q\otimes J_r)P_{m,(i,j),(k,\ell)}=P_{m,(i,j),(k,\ell)}(I_q\otimes J_r)$. 
Also, the case $K=J_q\otimes I_r$ and $P_{m,(i,j),(k,\ell)}$ or $P'_{m,(i,j),(k,\ell)}$ are similarly proven.  
\end{proof} 
We then determine the spectrum of $A_{[f],\text{MOSLS}}$ under the assumption that $A_{[f],\text{MOLS}}$ and $B$ commute. 
\begin{theorem}\label{thm:ev1}
Let $L_k$ {\rm (}$k\in [f]${\rm )} be mutually orthogonal  Sudoku Latin squares of order $qr$ of type $(q,r)$ on the set $[qr]$ of symbols $1,2,\ldots,qr$.  
If $A_{[f],\text{MOLS}}$ and $B$ commute, then the spectrum of $A_{[f],\text{MOSLS}}$ is 
\begin{align*}
\{&[(q-1)(r-1)+(qr-1)(f+2)]^1,[(q-1) (r-1)+q r-2-f]^{q+r-2},\\
&[(q-1) (r-1)-2-f]^{(q-1)(r-1)},\\
&[q r-1-f]^{f(q-1)(r-1)}, [q r-q-1-f]^{(r-1)(q+f)}, [q r-r-1-f]^{(q-1)(r+f)},\\
& [-1-f]^{(q-1)(r-1)(qr-f)}, [-q-1-f]^{(r-1)(qr-q-f)},  [-r-1-f]^{(q-1)(qr-r-f)}
\}.
\end{align*}
\end{theorem}
\begin{proof}
By Lemmas~\ref{lem:srgev}, \ref{lem:srgmols}, the spectrum of $A_{[f],\text{MOLS}}$ is  
$$
\text{Spec}(A_{[f],\text{MOLS}})=\{[(f+2)(qr-1)]^1,[qr-f-2]^{(f+2)(qr-1)},[-f-2]^{(qr-f-1)(qr-1)}\}. 
$$
The spectrum of $B$ is easy to see that 
$$
\text{Spec}(B)=\{[(q-1)(r-1)]^{qr},[1]^{q(q-1)r(r-1)},[-q+1]^{qr(r-1)},[-r+1]^{q(q-1)r}\}. 
$$
Write $\text{Spec}(A_{[f],\text{MOLS}})=\{[\theta_i]^{m_i} \mid i=1,2,3\}$ and $\text{Spec}(B)=\{[b_i]^{n_i} \mid i=1,2,3,4\}$ where $\theta_1>\theta_2>\theta_3$ and $b_1>b_2>0,b_3=-q+1,b_4=-r+1$. 
Note that $m_1=1$. 
Since $A_{[f],\text{MOLS}}$ and $B$ commute, we obtain 
$$
\text{Spec}(A_{[f],\text{MOSLS}})=\{[\theta_i+b_j]^{m_{i,j}}\mid i\in\{1,2,3\},j\in\{1,2,3,4\}\},  
$$
for some non-negative integer $m_{i,j}$.  
We determine $m_{i,j}$ in the following.  
Note that since $A_{[f],\text{MOLS}}$ and $B$ have the all ones vector as an eigenvector, $m_{1,1}=1,m_{1,2}=m_{1,3}=m_{1,4}=0$.  

Letting $\mathcal{E}(M,\theta)$ be the eigenspace of a real symmetric matrix $M$ corresponding to the eigenvalue $\theta$, we have  $\sum_{i=1}^3\mathcal{E}(A_{[f],\text{MOLS}}+B,\theta_i+b_j)=\mathcal{E}(B,b_j)$ and $\sum_{j=1}^{4}\mathcal{E}(A_{[f],\text{MOLS}}+B,\theta_i+b_j)=\mathcal{E}(A_{[f],\text{MOLS}},\theta_i)$. 
Comparing the dimension, we have 
\begin{align}
\sum_{i=1}^3 m_{i,j}&= n_j \text{ for }j\in\{1,2,3,4\},\label{eq:1}\\
\sum_{j=1}^{4} m_{i,j}&= m_i \text{ for }i\in\{1,2,3\}.\label{eq:2} 
\end{align}
For $j=1$, Equation~\eqref{eq:1} with $m_{1,1}=1$ reads as $m_{2,1}+m_{3,1}=qr-1$. 
Lemma~\ref{lem:quotient} shows that $m_{2,1}\geq q+r-2$ and $m_{3,1}\geq (q-1)(r-1)$. 
Therefore $m_{2,1}= q+r-2$ and $m_{3,1}= (q-1)(r-1)$. 
Since the graph with adjacency matrix $A_{[f],\text{MOLS}}$ is $((f + 2) (qr - 1) + (q - 1)(r-1))$-regular,
$$
{\rm tr}(A_{[f],\text{MOLS}}+B)^2=\sum_{i=1}^3 \sum_{j=1}^4(\theta_i+b_j)^2m_{i,j}=q^2r^2((f + 2) (qr - 1) + (q - 1)(r-1)).
$$
Furthermore, it follows from counting the triangles that 
\begin{itemize}
\item the number of triangles with three edges corresponding to $A_{[f],\text{MOLS}}$ is $q^2r^2(f+2)(qr-1)(qr-2+f(f+1))$, 
\item the number of triangles with two edges corresponding to $A_{[f],\text{MOLS}}$ and one edge $B$ is $3(f+1)(f+2)(q-1)(r-1)q^2r^2$,
\item the number of triangles with one edge corresponding to $A_{[f],\text{MOLS}}$ and two edges $B$ is $3q^2r^2(q-1)(r-1)(q+r-4)$,  
\item the number of triangles with three edges corresponding to $B$ is $q^2(q-1)(q-2)r^2(r-1)(r-2)$.
\end{itemize} 
Thus we have 
\begin{align*}
{\rm tr}(A_{[f],\text{MOLS}}+B)^3&=\sum_{i=1}^3\sum_{j=1}^4(\theta_i+b_j)^3m_{i,j}\\
&=q^2r^2(f+2)(qr-1)(qr-2+f(f+1))+3(f+1)(f+2)q^2r^2(q-1)(r-1)\\
&+3q^2r^2(q-1)(r-1)(q+r-4)+q^2(q-1)(q-2)r^2(r-1)(r-2).
\end{align*}

Solving Equations~\eqref{eq:1} and \eqref{eq:2} with the above yields that 
\begin{align*}
%&m_{1,1}=1, m_{1,2}=0,m_{1,3}=0,\\
&m_{2,1}=q+r-2, m_{2,2}=f(q-1)(r-1),\\
&m_{2,3}=(r-1)(q+f),m_{2,4}=(q-1)(r+f),\\
&m_{3,1}= (q-1)(r-1), m_{3,2}=(q-1)(r-1)(qr-f),\\
&m_{3,3}=(r-1)(qr-q-f),m_{3,4}=(q-1)(qr-r-f). \qedhere
\end{align*}

\end{proof}

\begin{example}
The following are orthogonal Sudoku Latin squares of order $4$:  
$$
L_1=\begin{array}{Ic|cIc|cI} \bhline{2pt}
 1 & 2 & 4 & 3 \\ \hline
 3 & 4 & 2 & 1 \\ \bhline{2pt}
 4 & 3 & 1 & 2 \\ \hline
 2 & 1 & 3 & 4 \\ \bhline{2pt}
\end{array},\quad L_2=\begin{array}{Ic|cIc|cI} \bhline{2pt}
 1 & 4 & 3 & 2 \\ \hline
 3 & 2 & 1 & 4 \\ \bhline{2pt}
 4 & 1 & 2 & 3 \\ \hline
 2 & 3 & 4 & 1 \\ \bhline{2pt}
\end{array}.
$$
The adjacency matrix $A_{[2],\text{MOSLS}}$ of the graph of orthogonal Sudoku Latin squares $L_1,L_2$ is 
$$
\left(
\begin{array}{cccccccccccccccc}
 0 & 1 & 1 & 1 & 1 & 1 & 1 & 1 & 1 & 1 & 1 & 1 & 1 & 0 & 0 & 1 \\
 1 & 0 & 1 & 1 & 1 & 1 & 1 & 1 & 1 & 1 & 1 & 1 & 0 & 1 & 1 & 0 \\
 1 & 1 & 0 & 1 & 1 & 1 & 1 & 1 & 1 & 1 & 1 & 1 & 0 & 1 & 1 & 0 \\
 1 & 1 & 1 & 0 & 1 & 1 & 1 & 1 & 1 & 1 & 1 & 1 & 1 & 0 & 0 & 1 \\
 1 & 1 & 1 & 1 & 0 & 1 & 1 & 1 & 1 & 0 & 0 & 1 & 1 & 1 & 1 & 1 \\
 1 & 1 & 1 & 1 & 1 & 0 & 1 & 1 & 0 & 1 & 1 & 0 & 1 & 1 & 1 & 1 \\
 1 & 1 & 1 & 1 & 1 & 1 & 0 & 1 & 0 & 1 & 1 & 0 & 1 & 1 & 1 & 1 \\
 1 & 1 & 1 & 1 & 1 & 1 & 1 & 0 & 1 & 0 & 0 & 1 & 1 & 1 & 1 & 1 \\
 1 & 1 & 1 & 1 & 1 & 0 & 0 & 1 & 0 & 1 & 1 & 1 & 1 & 1 & 1 & 1 \\
 1 & 1 & 1 & 1 & 0 & 1 & 1 & 0 & 1 & 0 & 1 & 1 & 1 & 1 & 1 & 1 \\
 1 & 1 & 1 & 1 & 0 & 1 & 1 & 0 & 1 & 1 & 0 & 1 & 1 & 1 & 1 & 1 \\
 1 & 1 & 1 & 1 & 1 & 0 & 0 & 1 & 1 & 1 & 1 & 0 & 1 & 1 & 1 & 1 \\
 1 & 0 & 0 & 1 & 1 & 1 & 1 & 1 & 1 & 1 & 1 & 1 & 0 & 1 & 1 & 1 \\
 0 & 1 & 1 & 0 & 1 & 1 & 1 & 1 & 1 & 1 & 1 & 1 & 1 & 0 & 1 & 1 \\
 0 & 1 & 1 & 0 & 1 & 1 & 1 & 1 & 1 & 1 & 1 & 1 & 1 & 1 & 0 & 1 \\
 1 & 0 & 0 & 1 & 1 & 1 & 1 & 1 & 1 & 1 & 1 & 1 & 1 & 1 & 1 & 0 \\
\end{array}
\right).
$$
Then the spectrum of $A_{[2],\text{MOSLS}}$ is  
\begin{align*}
\text{Spec}(A_{[2],\text{MOSLS}})&=\{[13]^{1},[1]^{4},[-1]^{8},[-3]^{3}\}. 
\end{align*}
\end{example}

\begin{example}
The following is a  Sudoku Latin squares of order $6$ of type $(2,3)$:  
$$
L=
\begin{array}{Ic|c|cIc|c|cI} \bhline{2pt}
 1 & 2 & 3 & 4 & 5 & 6 \\ \hline
 4 & 5 & 6 & 1 & 2 & 3 \\ \bhline{2pt}
 2 & 3 & 1 & 5 & 6 & 4 \\ \hline
 5 & 6 & 4 & 2 & 3 & 1 \\ \bhline{2pt}
 3 & 1 & 2 & 6 & 4 & 5 \\ \hline
 6 & 4 & 5 & 3 & 1 & 2 \\ \bhline{2pt}
\end{array}.
$$
Let $A$ be the adjacency matrix of the graph of the  Sudoku Latin square $L$ of order $6$. 
Then the spectrum of $A$ is  
\begin{align*}
\text{Spec}(A)&=\{[17]^1,[-5]^ 2, [-4]^6, [-2]^{10}, [-1]^2, [1]^4,[2]^6, [4]^2, [5]^3 \}. 
\end{align*}
\end{example}

\section{Construction of mutually orthogonal  Sudoku Latin squares}\label{sec:const}
\subsection{Using finite fields}
Let $p$ be a prime number, $m,n$ positive integers and set  $q=p^{m},r=p^n$.  
Let $f(t)$ be a irreducible polynomial of degree $m+n$ over $\mathbb{Z}_p$ and $\mathbb{F}_{qr}=\mathbb{Z}_p[t]/(f(t))$ be the finite field of $qr$ elements.

For $a\in \mathbb{F}_{qr}$, define $L_a=(x-ay)_{x,y\in \mathbb{F}_{qr}}$. 
It is well-known that $L_a$ ($a\in\mathbb{F}_{qr}\setminus\{0\}$) form mutually orthogonal Latin squares of order $qr$. 

Write 
\begin{align*}
&\left\{\sum_{i=n}^{m+n-1} a_i t^i \mid a_i\in \mathbb{Z}_p\right\}=\{f_1(t)=0,f_2(t),\ldots,f_{r}(t)\},\\
&\left\{\sum_{i=m}^{m+n-1} a_i t^i \mid a_i\in \mathbb{Z}_p\right\}=\{g_1(t)=0,g_2(t),\ldots,g_{q}(t)\}.
\end{align*}

Set 
\begin{align*}
X_k&=\left\{f_k(t)+\sum_{i=0}^{n-1}a_i t^i \mid a_i\in\mathbb{Z}_p\right\},\\
Y_\ell&=\left\{g_\ell(t)+\sum_{i=0}^{m-1}a_i t^i \mid a_i\in\mathbb{Z}_p\right\},
\end{align*}
for $k\in\{1,2,\ldots,r\}$ and for $\ell\in\{1,2,\ldots,q\}$. We regard $L_a$ as a block matrix with rows  according to the partition $\mathbb{F}_{qr}=X_1\cup\cdots\cup X_{r}$ and the columns according to the partition $\mathbb{F}_{qr}=Y_1\cup\cdots\cup Y_{q}$. 
\begin{lemma}
If $a$ is a polynomial of degree $n$, then $L_a$ is a Sudoku Latin square. 
%then the entries of each block in $L_a$ all different. 
\end{lemma} 
\begin{proof}
Consider a block corresponding to $X_k\times Y_\ell$ for $k,\ell\in\{1,\ldots,r\},\ell\in\{1,\ldots,q\}$.
Let $x,x'\in X_k,y,y'\in Y_\ell$. 
Assume that the $(x,y)$-entry equals to the $(x',y')$-entry. 
By definition, $x-ay=x'-ay'$, and thus $x-x'=a(y-y')$ holds. On the one hand $x-x'$ is a polynomial of degree at most $n-1$ in $t$, and on the other hand $a(y-y')$ is a polynomial of degree at least $n$ in $t$ provided that $y\neq y'$. 
Therefore $y=y'$ holds, and thus $x=x'$ follows. Therefore, all the entries of the block are different.   
\end{proof}
Thus we have the following result:  
For any prime powers $q=p^m,r=p^n$, there exist $\max(p^m(p-1),p^n(p-1))$ mutually orthogonal  Sudoku Latin squares of order $qr$ of type $(q,r)$. 
We then claim the resulting mutually orthogonal  Sudoku Latin squares are block permutational, that is,  satisfy the assumption of Lemma~\ref{lem:1}. 
For this, we need the following lemmas. 
\begin{lemma}\label{lem:adj1}
Let $a$ be a polynomial of degree $n$. 
For $(i,j),(i',j')\in[r]\times [q]$, the $(i,j)$-block $L[i,j]$ and $(i',j')$-block $L[i',j']$ of $L_a$ satisfy that  $L[i,j]=P^{-1}L[i',j']Q$ for some permutation matrices $P,Q$. 
\end{lemma}
\begin{proof} 
There uniquely exist $A,B\in \{\sum_{i=0}^{n-1}a_i t^i \mid a_i\in\mathbb{Z}_p\}$ such that 
$$
A-aB=(f_i-f_{i'})-a(g_j-g_{j'}).  
$$ 
Define bijections $\sigma, \tau$ on $[q]$ as 
$$
f_{\sigma(\alpha)}=f_{\alpha}+A, g_{\tau(\beta)}=g_{\beta}+B. 
$$
for $\alpha\in\{1,2,\ldots,r\},\beta\in\{1,2,\ldots,q\}$. 
Define $P=(\delta_{i,\sigma(j)})_{i,j=1}^{{r}},Q=(\delta_{i,\tau(j)})_{i,j=1}^q$
Then it is readily shown that 
$$
(f_i+g_{\alpha})-a(f_j+g_{\beta})=(f_{i'}+g_{\sigma(\alpha)})-a(f_{j'}+g_{\tau(\beta)}), 
$$
which implies that 
$(L[i,j])_{\alpha,\beta}=(L[i',j'])_{\sigma(\alpha),\tau(\beta)}=(P^{-1}L[i',j']Q)_{\alpha,\beta}$. 
\end{proof}
Then by considering the transpose of Latin squares, the following theorem is obtained  from Lemmas~\ref{lem:adj1}, \ref{lem:adj2}.  

\begin{theorem}\label{thm:const}
Let $p$ be a prime and $m,n$ positive integers, and set $q=p^{m},r=p^n$, and $\ell=\max\{p^m(p-1),p^n(p-1)\}$. 
Then there exist mutually orthogonal  Sudoku Latin squares $L_1,\ldots,L_{\ell}$ of order $qr$ of type $(q,r)$ which are of block permutational. 
\end{theorem}

\subsection{A recursive construction}
A recursive construction for mutually orthogonal Sudoku Latins squares is established. The type $(q,r)$ for seed  Sudoku Latin squares are allowed the cases $(q,1)$ or $(1,r)$.   
\begin{theorem}\label{thm:rec}
Assume that for $i=1,2$, there exist orthogonal  Sudoku Latin squares of order $q_i r_i$ of type $(q_i,r_i)$ such that each satisfies the assumption of Theorem~\ref{thm:ev1}.  
Then there exists orthogonal  Sudoku Latin squares of order $q_1q_2r_1r_2$ of type $(q_1q_2,r_1r_2)$  
 which are of block permutational.
% such that satisfies the assumption of Theorem~\ref{thm:ev1}. 
\end{theorem}
\begin{proof}
Let $L_1^{(i)},L_2^{(i)}$ be such  Sudoku Latin squares of order $q_i r_i$ of type $(q_i,r_i)$ for $i=1,2$. 
Assume that the rows and the columns of $L_1^{(i)},L_2^{(i)}$ are decomposed into $X_1^{(i)},\ldots,X_{r_i}^{(i)}$ and $Y_1^{(i)},\ldots,Y_{q_i}^{(i)}$, respectively.  
It is easy to see that 
$$
M_i:=(((L_i^{(1)})_{x_1,y_1},(L_i^{(2)})_{x_2,y_2}))_{(x_1,x_2),(y_1,y_2)\in [q_1r_1]\times[q_2r_2]}\quad  (i=1,2)
$$ 
 are orthogonal Latin squares of order $q_1q_2r_1r_2$. 
We then show that each $M_i$ is a  Sudoku of type $(q_1q_2,r_1r_2)$ satisfying  the assumption of Theorem~\ref{thm:ev1}.  
Note that the indices of blocks are $(i_1,i_2)$, $i_1\in[r_1],i_2\in [r_2]$, for the rows and $(j_1,j_2)$, $j_1\in[q_1],j_2\in [q_2]$, for the columns. 

First, the entries of $M_i$ on the rows in $X_{i_1}^{(1)}\times X_{i_2}^{(2)}$ where $i_1,i_2\in[r_1]$ and on the columns in $Y_{j_1}^{(1)}\times Y_{j_2}^{(2)}$ where $j_1,j_2\in[q_1]$ are 
$$
((L_i^{(1)})_{x_1,y_1},(L_i^{(2)})_{x_2,y_2})
$$
where  $ (x_1,x_2)\in X_{i_1}^{(1)}\times X_{i_2}^{(2)},(y_1,y_2)\in Y_{j_1}^{(1)}\times Y_{j_2}^{(2)}$. 
Then $(x_1,y_1)$ runs over the set $X_{i_1}^{(1)}\times Y_{j_1}^{(1)}$, and thus $(L_i^{(1)})_{x_1,y_1}$ takes all the elements in $[q_1r_1]$ exactly once because of the  Sudoku property for $L_1$. 
Similarly, the same is true for $(L_i^{(2)})_{x_2,y_2}$. 
Therefore, $M_i$ is a  Sudoku of type $(q_1q_2,r_1r_2)$.   

Next, let $N^{(1)}_{i_1,j_1},N^{(2)}_{i_2,j_2}$ be the $(i_1,j_1)$-block of $L_1^{(i)}$ and $(i_2,j_2)$-block of  $L_2^{(i)}$ respectively and $P^{(i)},Q^{(i)}$ be the permutation matrices of order $q_i,r_i$ respectively such that $N^{(1)}_{k,\ell}=(P^{(1)})^{-1}N^{(1)}_{1,1}Q^{(1)},N^{(2)}_{k,\ell}=(P^{(2)})^{-1}N^{(2)}_{1,1}Q^{(2)}$. 
Then, letting $\tilde{N}_{(i_1,i_2),(j_1,j_2)}^{(i)}$ be the $((i_1,i_2),(j_1,j_2))$-block of $M_i$, we have 
$$
\tilde{N}_{(i_1,i_2),(j_1,j_2)}^{(i)}=(P^{(1)}\otimes P^{(2)})^{-1}\tilde{N}_{(1,1),(1,1)}^{(i)}(Q^{(1)}\otimes Q^{(2)}).
$$
Therefore, $M_i$ is of block permutational.%satisfies the assumption of Theorem~\ref{thm:ev1}.  
\end{proof}

Combining Theorems~\ref{thm:const} and \ref{thm:rec}, we obtain: 
\begin{theorem}
Let $p_1,\ldots,p_k$ be distinct primes, $m_1,n_1,\ldots,m_k,n_k$ be non-negative integers such that either $m_i$ or $n_i$ is positive.  
Then there exist mutually orthogonal  Sudoku Latin squares $L_1,\ldots,L_f$ of order $\prod_{i=1}^k p_i^{m_i+n_i}$ of type $(\prod_{i=1}^k p_i^{m_i},\prod_{i=1}^k p_i^{n_i})$  which are of block permutational, 
%satisfying  the assumption of Theorem~\ref{thm:ev1}, 
where $f=\min\{\ell(p_1,m_1,n_1),\ldots,\ell(p_k,m_k,n_k)\}$ where 
$$
\ell(p,m,n)=\begin{cases}
\max\{p^m(p-1),p^n(p-1)\} & \text{ if }m>0 \text{ and }n>0,\\
\max\{p^m(p-1),p^n-1\} & \text{ if }m>0 \text{ and }n=0,\\
\max\{p^m-1,p^n(p-1)\} & \text{ if }m=0 \text{ and }n>0.
\end{cases}
$$
\end{theorem}
\begin{table}[htb]
\centering
  \caption{Lower bounds on the maximum number of $f$}
  \begin{tabular}{|c||c|c|}  \hline
    order & type & $f$  \\ \hline \hline
    2 & (1,2) & 1  \\ \hline
    3 & (1,3) & 2  \\  \hline
    4 & (1,4) & 3  \\ \cline{2-2} \cline{3-3}
     & (2,2) & 2  \\ \hline
    5 & (1,5) & 4  \\  \hline
    6 & (1,6) & 1  \\ \cline{2-2} \cline{3-3}
     & (2,3) & 1  \\ \hline
7 & (1,7) & 6  \\  \hline
    8 & (1,8) & 7  \\ \cline{2-2} \cline{3-3}
     & (2,4) & 4  \\ \hline
 9 & (1,9) & 8  \\ \cline{2-2} \cline{3-3}
     & (3,3) & 6  \\ \hline
 10 & (1,10) & $\geq 2$  \\ \cline{2-2} \cline{3-3}
     & (2,5) & $\geq 1$  \\ \hline
11 & (1,11) & 10  \\  \hline
 12 & (1,12) & $\geq 5$  \\ \cline{2-2} \cline{3-3}
     & (2,6) & $\geq 2$  \\ \cline{2-2} \cline{3-3}
     & (3,4) & $\geq 2$  \\ \hline
  \end{tabular}
\end{table}

\section{Cycle switching for a  Sudoku Latin square}
In Sections~\ref{sec:graph} and \ref{sec:const}, we constructed  Sudoku Latin squares whose graph spectrum is explicitly determined. In this section, we apply the method of cycle switching \cite{W} to  Sudoku Latin squares and determine the graph spectra. 

Let $L$ be a Latin square of order $n$ on the symbol $[n]=\{1,2,\ldots,n\}$. 
Regard a $\{r,s\}\times [n]$ subarray of $L$ as a permutation $\sigma:L_{r,i} \mapsto L_{s,i}$ for $i\in [n]$. 
Consider the cycle decomposition of $\sigma=\sigma_1\cdots \sigma_t$. 
Assume $\sigma_1$ involves the set $C$ of columns. 
A cycle switching of $L$ with respect to $\sigma_1$ is a new Latin square $L'$ obtained in the following way:
\begin{align*}
L'_{ij}&=\begin{cases}
L_{sj} & \text{ if }i=r \text{ and } j\in C, \\
L_{rj} & \text{ if }i=s \text{ and } j\in C, \\
L_{ij} & \text{ otherwise}. 
\end{cases}
\end{align*}

A cycle switching for a Sudoku Latin square does not necessarily provide a Sudoku Latin square, as described below. 
\begin{remark}
The following square $L_1$ is a Sudoku Latin square of order $4$ and $L_2$ is obtained from $L_1$ by switching the entries $1,4$ in the second and third rows:   
$$
L_1=\begin{array}{|c|c|c|c|}\hline
 1 & 2 & 3 & 4 \\ \hline
 3 & 4 & 1 & 2 \\ \hline
 2 & 1 & 4 & 3 \\ \hline
 4 & 3 & 2 & 1 \\ \hline
\end{array},\quad L_2=\begin{array}{|c|c|c|c|}\hline
 1 & 2 & 3 & 4 \\ \hline
 3 & 1 & 4 & 2 \\ \hline
 2 & 4 & 1 & 3 \\ \hline
 4 & 3 & 2 & 1 \\ \hline
\end{array}.
$$
Then $L_2$ is a Latin square, but not a Sudoku Latin square. 
Therefore the cycle switching does not necessarily preserve the property of Sudoku. 
\end{remark}

Let $L$ be a  Sudoku Latin square of order $qr$ of type $(q,r)$ with $q,r\geq 2$ which are of block permutational. %satisfying the assumption of Theorem~\ref{thm:ev1}. 
Let $k_1,k_2$ be distinct entries in a column of  $B_{1,1}$ in $L$, and set $\sigma=(k_1,k_2)$ to be a transposition.    
We then apply $\sigma$ on the first row blocks $B_{1,1},B_{1,2},\ldots,B_{1,q}$, and set the resulting Sudoku Latin square to be $L'$\footnote{This is certainly cycle switching if we regard a Latin square as the set of triples $(i,j,L_{ij})$ and then }.  
In this section, we determine the spectrum of the graph $G_{L'}$ and claim the spectrum of $G_{L'}$ differs from that of $G_L$. 
Thus we conclude that they are not isomorphic. 

Write $A_L$ and $A_{L'}$ as the adjacency matrices of the Latin square graphs. 
We prove the following theorem.  
\begin{theorem}
The spectrum of $A_{L'}+B$ is 
\begin{align*}
(\mathrm{Spec}(A_{L}+B)\setminus X)\cup Y 
\end{align*}
as a multiset, where 
\begin{align*}
\mathrm{Spec}(A_{L}+B)=&\{[(q-1)(r-1)+3qr-3]^1,[(q-1) (r-1)+q r-3]^{q+r-2},\\
&[(q-1) (r-1)-3]^{(q-1)(r-1)},[q r-2]^{f(q-1)(r-1)},\\
& [q r-q-2]^{(q+1)(r-1)}, [q r-r-2]^{(q-1)(r+1)},\\
& [-2]^{(q-1)(r-1)(qr-1)}, [-q-2]^{(r-1)(qr-q-1)},  [-r-2]^{(q-1)(qr-r-1)}
\},\\
X=&\{-2,-r-2,qr-2,qr-r-2\},\\
Y=&\left\{\frac{1}{2} \left((q-1) r-4\pm\sqrt{\left(q^2+1\right) r^2\pm2 q(r-2) \sqrt{r^2+4 r-4}}\right)\right\}.
\end{align*}
\end{theorem}
\begin{proof}
If $x,y\in B_{1,1}$ or  $x,y\in \bigcup_{\substack{2\leq k\leq r \\ 2\leq \ell\leq q}}B_{k,\ell}$, then 
\begin{align*}
(A_{L'}-A_L)_{x,y}&=0.
\end{align*} 
Since $B_{1,\ell}$ is permutaionally equivalent to $B_{1,1}$, the submatrix of $A_{L'}-A_L$ with rows restricted to $B_{1,1}$ and columns restricted to $B_{1,\ell}$ has two non-zero column vectors such that one is negative to the other for $\ell\in\{2,\ldots,q\}$.  
Moreover,  the set of non-zero column vectors with respect to $B_{1,\ell}$ coincides with those with respect to  $B_{1,\ell'}$ for  $\ell,\ell'\in\{2,\ldots,q\}$.    

Thus the matrix $Y=A_{L'}-A_L$ is of the following form:
$$
Y=\begin{pmatrix}
 0 & Z \\ 
 Z^\top & 0
 \end{pmatrix},
$$ 
where $Z$ is a $q^2r\times (q^2r^2-q^2r)$ $(0,\pm1)$-matrix of rank $1$.  
Then there exist column vectors $u,v$ of $Y$ with $u^\top u=2q(r-1), v^\top v=2q, v^\top v=0$ such that $Y=UV^\top$ holds, where 
$$
U=\begin{pmatrix} u & v \end{pmatrix}, \quad V=\begin{pmatrix} v & u \end{pmatrix}. 
$$
Note that $u,v$ can be written as 
\begin{align}
u&=\sum_{i=2}^r\sum_{j=1}^{q} e_i^{(q)}\otimes e_j^{(r)}\otimes (e_{\psi_1(i,j)}^{(q)}-e_{\psi_2(i,j)}^{(q)})\otimes e_{\psi_3(i,j)}^{(r)},\label{eq:u}\\
v&=\sum_{j=1}^{q} e_1^{(q)}\otimes e_j^{(r)}\otimes (e_{\psi_1(1,j)}^{(q)}-e_{\psi_2(1,j)}^{(q)})\otimes e_{\psi_3(1,j)}^{(r)},\label{eq:v}
\end{align}
where $\psi_1,\psi_2$ are functions from $[q]\times [r]$ to $[q]$ such that $\psi_1(i,j)\neq \psi_2(i,j)$ for any $i,j$ and $\psi_3$ is a functions from $[q]\times [r]$ to $[r]$. 

Since $A_{L'}=A_L+UV^\top$ and by matrix determinant lemma, we obtain 
\begin{align}
\det(A_{L'}-tI_{q^2r^2})&=\det(A_{L}-tI_{q^2r^2}+UV^\top)\nonumber\\
&=\det(A_{L}-tI_{q^2r^2})\det(I_2+V^\top (A_{L}-tI_{q^2r^2})^{-1}U).\label{eq:cs01}
\end{align}
Let $A_L$ have the spectral decomposition as $A_L=\sum_{i=1}^3\theta_i E_i$, where $\theta_i$ are the same as in the proof of Theorem~\ref{thm:ev1} with $f=1$.  
Since $E_1u=E_1v={\bf 0}$ which follow from \eqref{eq:u} and \eqref{eq:v}, 
\begin{align*}
V^\top (A_{L}-tI_{{q^2r^2}})^{-1}U&=V^\top \left(\sum_{i=1}^3 \frac{1}{\theta_i-t}E_i\right)U\\
&=\sum_{i=1}^3 \frac{1}{\theta_i-t}V^\top E_iU\\
&=\sum_{i=1}^3 \frac{1}{\theta_i-t}\begin{pmatrix}  v^\top E_i u& v^\top E_i v \\    u^\top E_i u &  u^\top E_i v\end{pmatrix}\\
&=\sum_{i=2}^3 \frac{1}{\theta_i-t}\begin{pmatrix}  v^\top E_i u& v^\top E_i v \\    u^\top E_i u &  u^\top E_i v\end{pmatrix}.
\end{align*}
Therefore Equation~\eqref{eq:cs01} is 
\begin{align}\label{eq:cs02}
\eqref{eq:cs01}=\det(A_{L}-tI_{q^2r^2})\det\left(I_2+\sum_{i=2}^3 \frac{1}{\theta_i-t}\begin{pmatrix}  v^\top E_i u& v^\top E_i v \\    u^\top E_i u &  u^\top E_i v\end{pmatrix}\right).
\end{align} 
Since $A_{L'}$ and $A_L$ are the adjacency matrices of strongly regular graphs with the same parameters,  we have $\det(A_{L}-tI_{q^2r^2})=\det(A_{L'}-tI_{q^2r^2})$. 
Combining this with Equation~\eqref{eq:cs02} yields that 
$$
\det\left(I_2+\sum_{i=2}^3 \frac{1}{\theta_i-t}\begin{pmatrix}  v^\top E_i u& v^\top E_i v \\    u^\top E_i u &  u^\top E_i v\end{pmatrix}\right)=1, 
$$
which gives rise to.  
Therefore, by using the formula that $u^\top E_2 u+u^\top E_3 u=2q(r-1),v^\top E_2 v+v^\top E_3 v=2q$, which follow from \eqref{eq:u} and \eqref{eq:v} and $u^\top E_i v=v^\top E_i u$ for $i=2,3$, 
we obtain the following identity in a variable $t$: 
\begin{align*}
&-36 q + 36 q r + 24 q^2 r - 24 q^2 r^2 - 4 q^3 r^2 + 4 q^3 r^3 + 
 (18 r  - 6 q r^2) a_{1,1} - q r^2 a_{1,1}^2 - 
 6 q r a_{1,2}\\
 & + (6 q r^2  + 2 q^2 r^2  - 
 2 q^2 r^3 )a_{1,2} + (6 q r  - 2 q^2 r^2) a_{2,1} + 
 q r^2 a_{1,2} a_{2,1} \\
 &+ 
 t (-24 q + 24 q r + 8 q^2 r - 8 q^2 r^2 + (12 r  - 
    2 q r^2) a_{1,1} +(- 2 q r  + 2 q r^2) a_{1,2} + 
    2 q r a_{2,1})\\
    & +
 t^2 (-4 q + 4 q r + 2 r a_{2,1})=0,
\end{align*}
where $a_{1,1}=v^\top E_2 u, a_{1,2}=v^\top E_2 v, a_{2,1}=u^\top E_2 u$.  
Then, either of the following holds: 
\begin{enumerate}
\item $
\begin{pmatrix}  v^\top E_2 u& v^\top E_2 v \\    u^\top E_2 u &  u^\top E_2 v\end{pmatrix}=\begin{pmatrix} -\frac{2q(r-1)}{r} & \frac{2q}{r} \\ \frac{2q(r-1)^2}{r} & -\frac{2q(r-1)}{r} \end{pmatrix},\quad \begin{pmatrix}  v^\top E_3 u& v^\top E_3 v \\    u^\top E_3 u &  u^\top E_3 v\end{pmatrix}=\begin{pmatrix}\frac{2q(r-1)}{r}& \frac{2q(r-1)}{r}\\ \frac{2q(r-1)}{r} & \frac{2q(r-1)}{r} \end{pmatrix}
$.
\item $
\begin{pmatrix}  v^\top E_2 u& v^\top E_2 v \\    u^\top E_2 u &  u^\top E_2 v\end{pmatrix}=\begin{pmatrix} -\frac{2q(r-1)}{r} & \frac{2q(r-1)}{r}\\ \frac{2q(r-1)}{r} & -\frac{2q(r-1)}{r} \end{pmatrix},\quad \begin{pmatrix}  v^\top E_3 u& v^\top E_3 v \\    u^\top E_3 u &  u^\top E_3 v\end{pmatrix}=\begin{pmatrix} \frac{2q(r-1)}{r} &  \frac{2q}{r} \\  \frac{2q(r-1)^2}{r} & \frac{2q(r-1)}{r} \end{pmatrix}
$.
\end{enumerate} 
Then, corresponding to the above, we have the following:  
\begin{enumerate}
\item $E_2(u-v)=u-v,E_2(u+(r-1)v)={\bf 0},E_3(u-v)={\bf 0}, E_3(u+(r-1)v)=u+(r-1)v$, that is $E_2u=\frac{r-1}{r}(u-v),E_2v=-\frac{1}{r}(u-v),E_3u=E_3v=\frac{1}{r}(u+(r-1)v)$.
\item $E_2(u+v)={\bf 0},E_2(u-(r-1)v)=u-(r-1)v,E_3(u+v)=u+v, E_3(u-(r-1)v)={\bf 0}$, that is $E_2u=-\frac{1}{r}(u-(r-1)v),E_2v=\frac{1}{r}(u-(r-1)v),E_3u=\frac{r-1}{r}(u+v),E_3v=\frac{1}{r}(u+v)$. 
\end{enumerate}
Indeed, for (i), $(u^\top-v^\top)E_3(u-v)=u^\top E_3 u-u^\top E_3 v-v^\top E_3 u+v^\top E_3 v=0$, and thus $E_3(u-v)={\bf 0}$ by $E_3$ being positive semidefnite. 
Since $I=\sum_{i=1}^3 E_i$ and $E_1u=E_1v={\bf 0}$, we have $E_2(u-v)=u-v$. The others are similarly shown.

Next we calculate $\det(A_{L'}+B-tI_{q^4})$. 
By matrix determinant lemma again, we obtain: 
\begin{align}
\det(A_{L'}+B-tI_{q^4})&=\det(A_{L}+B-tI_{q^4}+UV^\top)\nonumber\\
&=\det(A_{L}+B-tI_{q^4})\det(I_2+V^\top (A_{L}+B-tI_{q^4})^{-1}U).\label{eq:cs1}
\end{align}
Let $B$ have the spectral decomposition as $B=\sum_{i=1}^{4}b_i F_i$. 
Note that 
\begin{align*}
F_1&=\frac{1}{qr}I_{qr}\otimes J_{qr},\\
F_2&=I_{qr}\otimes(I_q-\frac{1}{q}J_q)\otimes (I_r-\frac{1}{r}J_r),\\
F_3&=\frac{1}{q}I_{qr}\otimes J_q \otimes(I_r-\frac{1}{r}J_r), \\ 
F_4&=\frac{1}{r}I_{qr}\otimes(I_q-\frac{1}{q}J_q)\otimes J_{r}. 
\end{align*}
Since
\begin{align*}
A_L+B=A_l\cdot I+I\cdot B=\sum_{i=1}^3 \theta_i E_i\sum_{j=1}^4 F_j+\sum_{i=1}^3 E_i \sum_{j=1}^4b_j  F_j=\sum_{i=1}^3\sum_{j=1}^4 (\theta_i+b_j) E_iF_j
\end{align*}
%Since ${A_L+B=\sum_{i=1}^3\sum_{j=1}^4 (\theta_i+b_j) E_iF_j}$
 and $E_1u=F_1u=E_1v=F_1v=F_3u=F_3v={\bf 0}$ which follow from \eqref{eq:u} and \eqref{eq:v}, 
\begin{align}
V^\top (A_{L}+B-tI_{q^4})^{-1}U&=V^\top (\sum_{i=1}^3 \sum_{j=1}^4 \frac{1}{\theta_i+b_j-t}E_iF_j)U\nonumber\\
&=\sum_{i=1}^3 \sum_{j=1}^4 \frac{1}{\theta_i+b_j-t}V^\top E_iF_jU\nonumber\\
&=\sum_{i=1}^3\sum_{j=1}^4 \frac{1}{\theta_i+b_j-t}\begin{pmatrix}  v^\top E_iF_j u& v^\top E_iF_j v \\    u^\top E_iF_j u &  u^\top E_iF_j v\end{pmatrix}\nonumber\\
&=\sum_{i=2}^3\sum_{j=2,4} \frac{1}{\theta_i+b_j-t}\begin{pmatrix}  v^\top E_iF_j u& v^\top E_iF_j v \\    u^\top E_iF_j u &  u^\top E_iF_j v\end{pmatrix}\label{eq:cs0211}
\end{align}
Using the following fact: 
\begin{align*}
u^\top F_2 u=\frac{2q(r-1)^2}{r}, v^\top F_2 v=\frac{2q(r-1)}{r},u^\top F_2v=v^\top F_2u=0, \\
u^\top F_4 u=\frac{2q(r-1)}{r}, v^\top F_4 v=\frac{2q}{r}, u^\top F_4v=v^\top F_4u=0, 
\end{align*}
we have in the case (i)
\begin{align*}
\begin{pmatrix}  v^\top E_2F_2 u& v^\top E_2F_2 v \\    u^\top E_2F_{2} u &  u^\top E_2F_{2} v\end{pmatrix}=\begin{pmatrix} -\frac{2q(r-1)^2}{r^2} & \frac{2q(r-1)}{r^2} \\ \frac{2q(r-1)^3}{r^2} & -\frac{2q(r-1)^2}{r^2} \end{pmatrix},
\begin{pmatrix}  v^\top E_2F_{4} u& v^\top E_2F_{4} v \\    u^\top E_2F_{4} u &  u^\top E_2F_4 v\end{pmatrix}=\begin{pmatrix} -\frac{2q(r-1)}{r^2} & \frac{2q}{r^2} \\ \frac{2q(r-1)^2}{r^2} & -\frac{2q(r-1)}{r^2}\end{pmatrix},\\
\begin{pmatrix}  v^\top E_3F_2 u& v^\top E_3F_2 v \\    u^\top E_3F_2 u &  u^\top E_3F_2 v\end{pmatrix}=\begin{pmatrix} {\frac{2q(r-1)^2}{r^2}} & {\frac{2q(r-1)^2}{r^2}} \\ {\frac{2q(r-1)^2}{r^2}} & {\frac{2q(r-1)^2}{r^2}}\end{pmatrix},
\begin{pmatrix}  v^\top E_3F_{{4}} u& v^\top E_3F_{{4}} v \\    u^\top E_3F_{{4}} u &  u^\top E_3F_{{4}} v\end{pmatrix}=\begin{pmatrix} {\frac{2q(r-1)}{r^2}} & {\frac{2q(r-1)}{r^2}} \\ {\frac{2q(r-1)}{r^2}} & {\frac{2q(r-1)}{r^2}}\end{pmatrix}. 
\end{align*}
while we have in the case (ii)
\begin{align*}
\begin{pmatrix}  v^\top E_2F_2 u& v^\top E_2F_2 v \\    u^\top E_2F_{{2}} u &  u^\top E_2F_{{2}} v\end{pmatrix}=\begin{pmatrix} {-\frac{2q(r-1)^2}{r^2}} & {\frac{2q(r-1)^2}{r^2}} \\ {\frac{2q(r-1)^2}{r^2}} & {-\frac{2q(r-1)^2}{r^2}}\end{pmatrix},
\begin{pmatrix}  v^\top E_2F_{{4}} u& v^\top E_2F_{{4}} v \\    u^\top E_2F_{{4}} u &  u^\top E_2F_{{4}} v\end{pmatrix}=\begin{pmatrix} {-\frac{2q(r-1)}{r^2}} & {\frac{2q(r-1)}{r^2}} \\ {\frac{2q(r-1)}{r^2}} & {-\frac{2q(r-1)}{r^2}}\end{pmatrix},\\
\begin{pmatrix}  v^\top E_3F_2 u& v^\top E_3F_2 v \\    u^\top E_3F_2 u &  u^\top E_3F_2 v\end{pmatrix}=\begin{pmatrix} {\frac{2q(r-1)^2}{r^2}} & {\frac{2q(r-1)}{r^2}} \\ {\frac{2q(r-1)^3}{r^2}} & {\frac{2q(r-1)^2}{r^2}}\end{pmatrix},
\begin{pmatrix}  v^\top E_3F_{{4}} u& v^\top E_3F_{{4}} v \\    u^\top E_3F_{{4}} u &  u^\top E_3F_{{4}} v\end{pmatrix}=\begin{pmatrix} {\frac{2q(r-1)}{r^2}} & {\frac{2q}{r^2}} \\ {\frac{2q(r-1)^2}{r^2}} & {\frac{2q(r-1)}{r^2}}\end{pmatrix}. 
\end{align*}

Then Equation~\eqref{eq:cs0211} in either case (i), (ii) is 
\begin{align}
\eqref{eq:cs0211}&=\det\left(I_2+\sum_{i=2}^3{\sum_{j=2,4}} \frac{1}{\theta_i+{b_j}-t}\begin{pmatrix}  v^\top E_iF_j u& v^\top E_iF_j v \\    u^\top E_iF_j u &  u^\top E_iF_j v\end{pmatrix}\right)\nonumber\\
&=\frac{f(t)}{(t+2) (t+{r}+2) \left(t-{qr}+2\right) \left(t-{qr}+{r}+2\right)}\label{eq:cs03}
\end{align}
where 
\begin{align*}
%f(t)&=t^4+\left(-2 q^2+2 q+8\right) t^3+\left(q^4-3 q^3-11 q^2+12 q+24\right) t^2\\
%&+\left(q^5+3 q^4-12 q^3-20 q^2+24 q+32\right) t+2 q^5+6 q^4-20 q^3 -8 q^2+16q+16.
f(t)&={t^4+\left(-2 qr+2 r+8\right) t^3+\left(q^2 r^2-3 q r^2-12 q r+r^2+12 r+24\right) t^2}\\
&{+\left(q^2 r^3+4 q^2 r^2-q r^3-12 q r^2-24 q r+4 r^2+24 r+32\right) t}\\
&{+2 q^2 r^3+8 q^2 r^2-8 q^2 r+4 q^2-2 q r^3-12 q r^2-16 q r+4 r^2+16 r+16}.
\end{align*}

Substituting \eqref{eq:cs03} into \eqref{eq:cs1} yields that 
\begin{align}
\det(A_{L'}+B-tI_{{q^2r^2}})&=\frac{\det(A_{L}+B-tI_{{q^2r^2}})f(t)}{(t+2) (t+{r}+2) \left(t-{qr}+2\right) \left(t-{qr}+{r}+2\right)}.
\end{align}
Note that the polynomial $f(t)$ has the roots 
$$
t={\frac{1}{2} \left((q-1) r-4\pm\sqrt{\left(q^2+1\right) r^2\pm2 q(r-2) \sqrt{r^2+4 r-4}}\right)}.
$$
This completes the proof. 
\end{proof}

\begin{example}
The following are examples of Sudoku Latin squares of order $4$:  
$$
L_1=\begin{array}{Ic|cIc|cI}\bhline{2pt}
 1 & 2 & 3 & 4 \\ \hline
 3 & 4 & 1 & 2 \\ \bhline{2pt}
 2 & 1 & 4 & 3 \\ \hline
 4 & 3 & 2 & 1 \\ \bhline{2pt}
\end{array},\quad L_2=\begin{array}{Ic|cIc|cI} \bhline{2pt}
 1 & 2 & 3 & 4 \\ \hline
 3 & 4 & 2 & 1 \\ \bhline{2pt}
 2 & 1 & 4 & 3 \\ \hline
 4 & 3 & 1 & 2 \\ \bhline{2pt}
\end{array}
$$
Write the adjacency matrix of graphs of Sudoku Latin squares $L_i$ as $A_{L_i}$ ($i\in\{1,2\}$). 
Then the spectra of $A_{L_1}$ and $A_{L_2}$ are  
\begin{align*}
\text{Spec}(A_{L_1})&=\{[10]^{1},[-4]^{2},[-2]^{4},[2]^{3},[0]^{6}\},\\
\text{Spec}(A_{L_2})&=\{[10]^{1},[-4]^{1},[-1-\sqrt{5}]^{2},[-2]^{3},[2]^{2},[-1+\sqrt{5}]^{2},[0]^{5}\}. 
\end{align*}
\end{example}

\begin{example}
The following are examples of Sudoku Latin squares of order $6$ of type $(2,3)$:  

$$
L_1 = 
\begin{array}{Ic|c|cIc|c|cI} \bhline{2pt}
 1 & 2 & 3 & 4 & 5 & 6 \\ \hline
 4 & 5 & 6 & 1 & 2 & 3 \\ \bhline{2pt}
 2 & 3 & 1 & 5 & 6 & 4 \\ \hline
 5 & 6 & 4 & 2 & 3 & 1 \\ \bhline{2pt}
 3 & 1 & 2 & 6 & 4 & 5 \\ \hline
 6 & 4 & 5 & 3 & 1 & 2 \\ \bhline{2pt}
\end{array},\quad L_2 = 
\begin{array}{Ic|c|cIc|c|cI} \bhline{2pt}
 4 & 2 & 3 & 1 & 5 & 6 \\ \hline
 1 & 5 & 6 & 4 & 2 & 3 \\ \bhline{2pt}
 2 & 3 & 1 & 5 & 6 & 4 \\ \hline
 5 & 6 & 4 & 2 & 3 & 1 \\ \bhline{2pt}
 3 & 1 & 2 & 6 & 4 & 5 \\ \hline
 6 & 4 & 5 & 3 & 1 & 2 \\ \bhline{2pt}
\end{array}
$$
Write the adjacency matrix of graphs of Sudoku Latin squares $L_i$ as $A_{L_i}$ ($i\in\{1,2\}$). 
Then the spectra of $A_{L_1}$ and $A_{L_2}$ are  
\begin{align*}
\text{Spec}(A_{L_1})&=\{[17]^1,[-5]^2,[-4]^6,[-2]^{10},[-1]^2,[1]^4,[2]^6,[4]^2,[5]^3\},\\
\text{Spec}(A_{L_2})&=\left\{[17]^1,[-5]^2,[-4]^5,[-2]^{9},[-1]^2,[1]^4,[2]^5,[4]^1,[5]^3,\left[\frac{1}{2} \left(-1\pm\sqrt{45\pm4 \sqrt{17}}\right)\right]^1\right\}. 
\end{align*}

\end{example}

\begin{example}
The following are examples of Sudoku Latin squares of order $9$: 
$$
L_1=
\begin{array}{Ic|c|cIc|c|cIc|c|cI} \bhline{2pt}
 5 & 6 & 4 & 8 & 9 & 7 & 2 & 3 & 1 \\ \hline
 9 & 7 & 8 & 3 & 1 & 2 & 6 & 4 & 5 \\ \hline
 1 & 2 & 3 & 4 & 5 & 6 & 7 & 8 & 9 \\ \bhline{2pt}
 3 & 1 & 2 & 6 & 4 & 5 & 9 & 7 & 8 \\ \hline
 4 & 5 & 6 & 7 & 8 & 9 & 1 & 2 & 3 \\ \hline
 8 & 9 & 7 & 2 & 3 & 1 & 5 & 6 & 4 \\ \bhline{2pt}
 7 & 8 & 9 & 1 & 2 & 3 & 4 & 5 & 6 \\ \hline
 2 & 3 & 1 & 5 & 6 & 4 & 8 & 9 & 7 \\ \hline
 6 & 4 & 5 & 9 & 7 & 8 & 3 & 1 & 2 \\ \bhline{2pt}
\end{array},\quad
L_2=
\begin{array}{Ic|c|cIc|c|cIc|c|cI} \bhline{2pt}
 5 & 6 & 4 & 8 & 9 & 7 & 1 & 3 & 2 \\ \hline
 9 & 7 & 8 & 3 & 1 & 2 & 6 & 4 & 5 \\ \hline
 1 & 2 & 3 & 4 & 5 & 6 & 7 & 8 & 9 \\ \bhline{2pt}
 3 & 1 & 2 & 6 & 4 & 5 & 9 & 7 & 8 \\ \hline
 4 & 5 & 6 & 7 & 8 & 9 & 2 & 1 & 3 \\ \hline
 8 & 9 & 7 & 2 & 3 & 1 & 5 & 6 & 4 \\ \bhline{2pt}
 7 & 8 & 9 & 1 & 2 & 3 & 4 & 5 & 6 \\ \hline
 2 & 3 & 1 & 5 & 6 & 4 & 8 & 9 & 7 \\ \hline
 6 & 4 & 5 & 9 & 7 & 8 & 3 & 2 & 1 \\ \bhline{2pt}
\end{array}
$$

Write the adjacency matrix of graphs of Sudoku Latin squares $L_i$ as $A_{L_i}$ ($i\in\{1,2\}$). 
Then the spectra of $A_{L_1}$ and $A_{L_2}$ are  
\begin{align*}
\text{Spec}(A_{L_1})&=\{
 [28]^1,
 [10]^4,
 [7]^4,
 [-5]^{20},
 [4]^{16},
 [-2]^{32},
 [1]^4 
\},\\
\text{Spec}(A_{L_2})&=\left\{
 [28]^1, 
 [10]^4,
 [7]^3,
% [\frac{1}{2} \left(\sqrt{6 \left(\sqrt{17}+15\right)}+2\right)]^1,
% [\frac{1}{2} \left(\sqrt{6 \left(15\pm\sqrt{17}\right)}+2\right)]^1,
 [-5]^{19},
% [\frac{1}{2} \left(2-\sqrt{6 \left(\sqrt{17}+15\right)}\right)]^1,
 [4]^{15},
 [-2]^{31},
 [1]^4,
 \left[\frac{1}{2} \left(2\pm\sqrt{6 \left(15\pm\sqrt{17}\right)}\right)\right]^1 
\right\}. 
\end{align*}
\end{example}

\section*{Acknowledgments.}
Sho Kubota is supported by JSPS KAKENHI Grant Nnumber 20J01175, and 
Sho Suda is supported by JSPS KAKENHI Grant Number 18K03395.

%%%%%%%%%%%%%%%%%%%%%%%%%%%%%%%%%%%%%%%%%%%%%

\end{document}